\documentclass[11pt]{article}

\usepackage[dvipsnames,table]{xcolor}
\definecolor{c0}{HTML}{ffb000}
\definecolor{c1}{HTML}{fe6100}
\definecolor{c2}{HTML}{dc267f}
\definecolor{c3}{HTML}{785ef0}
\definecolor{c4}{HTML}{648fff}

\usepackage[margin=1.15in]{geometry}

\usepackage{amsmath,amssymb,amsthm,mathtools}
\usepackage{hyperref}
\hypersetup{
    colorlinks=true,      
    linkcolor=c2,       
    citecolor=c4,      
    filecolor=c4,    
    urlcolor=c4         
}
\usepackage[capitalize,nameinlink,noabbrev]{cleveref}
\crefformat{equation}{#2(#1)#3}
\crefmultiformat{equation}{#2(#1)#3}%
{ and~#2(#1)#3}{, #2(#1)#3}{ and~#2(#1)#3}

\usepackage[labelfont=bf,font=small]{caption}

\usepackage[backend=biber,backref,style=alphabetic,maxalphanames=4,maxcitenames=99,maxbibnames=99,giveninits=true,url=false,doi=false,isbn=false,eprint=false,date=year]{biblatex}
\DeclareSourcemap{
  \maps[datatype=bibtex, overwrite]{
    \map{
      \step[fieldset=editor, null]
    }
  }
}
\DefineBibliographyStrings{english}{%
  backrefpage = {cited on page},%
  backrefpages = {cited on pages}%
}

\renewcommand{\vec}[1]{\mathbf{#1}}
\newcommand{\T}{\mathsf{T}}
\newcommand{\R}{\mathbb{R}}
\newcommand{\C}{\mathbb{C}}
\newcommand{\Pol}{\Pi}
\newcommand{\Kry}{\mathcal{K}}

\newcommand{\lmin}{\lambda_{\min}}
\newcommand{\lmax}{\lambda_{\max}}
\newcommand{\dmu}{\,\mathrm{d}\mu(t)}

\newcommand{\ip}[2]{\langle #1, #2 \rangle}

\newcommand{\vvvert}{|\!|\!|}

\newtheorem{theorem}{Theorem}
\newtheorem*{theorem*}{Theorem}
\newtheorem{lemma}[theorem]{Lemma}

\theoremstyle{definition}
\newtheorem{remark}[theorem]{Remark}

\title{Optimal near-optimality bounds for the Lanczos method for matrix functions}
\author{Tyler Chen\thanks{NYU (\url{tyler.chen@nyu.edu})} \and David Persson\thanks{Flatiron Institute and NYU (\url{dup210@nyu.edu})}}
\date{}

\begin{document}
\maketitle

\begin{abstract}
Let $\vec{A}$ be Hermitian positive definite and let $\vec{f}_m$ denote the Lanczos approximation to $f(\vec{A})\vec{b}$.
We prove that if $f(z)$ or $f(z) / z$ is Stieltjes, then the $\vec{A}^\alpha$-norm error of the Lanczos approximation is within a factor $\tfrac{1}{2}(\kappa(\vec{A})^{E/2} + \kappa(\vec{A})^{-E/2})$ of the the best possible Krylov Subspace Method, where $\kappa(\vec{A})$ is the condition number of $\vec{A}$ and $E = \max\{\alpha,1-\alpha\}$. 
Our result strengthens and generalizes the upper bound of [Schweitzer; SIMAX, 46.3 (2025)].
Moreover, we prove that the constant $\tfrac{1}{2}(\kappa(\vec{A})^{E/2} + \kappa(\vec{A})^{-E/2})$ is optimal.
\end{abstract}

\section{Introduction}

Consider a Hermitian matrix $\vec{A}\in\C^{n\times n}$ with eigendecomposition $\sum_{i=1}^{n} \lambda_i \vec{w}_i\vec{w}_i^*
$, a vector $\vec{b}\in\C^n$, and scalar function $f$ defined on the eigenvalues of $\vec{A}$.
Computing 
\[
f(\vec{A})\vec{b}
,\qquad\text{where}\qquad
f(\vec{A}) := \sum_{i=1}^{n} f(\lambda_i) \vec{w}_i\vec{w}_i^*,
\]
is an important task in a diverse collection of applications including
solving systems of linear equations and least squares problems ($f(z) = 1/z$) \cite{hestenes_stiefel_52,greenbaum_97},
sampling from Gaussian distributions and other computations involving matrix square roots ($f(z) = z^{\pm 1/2}$) \cite{chow_saad_14,pleiss_jankowiak_eriksson_damle_gardner_20},
quantum dynamics and the solution of differential equations ($f(z) = \exp(z)$) \cite{druskin_knizhnerman_95,hochbruck_lubich_97,hochbruck_ostermann_10},
lattice quantum chromodynamics ($f(z) = \operatorname{sign}(z)$) \cite{eshof_frommer_lippert_schilling_vanderVorst_02},
Gaussian process regression and log-determinant computation ($f(z) = \log(z)$) \cite{ubaru_chen_saad_17,dong_eriksson_nickisch_bindel_wilson_17},
and fractional differential equations ($f(z) = z^{\gamma}$) \cite{ilic_turner_anh_08,ilic_turner_simpson_09}.

\begin{figure}
    \centering
    \includegraphics[scale=.6]{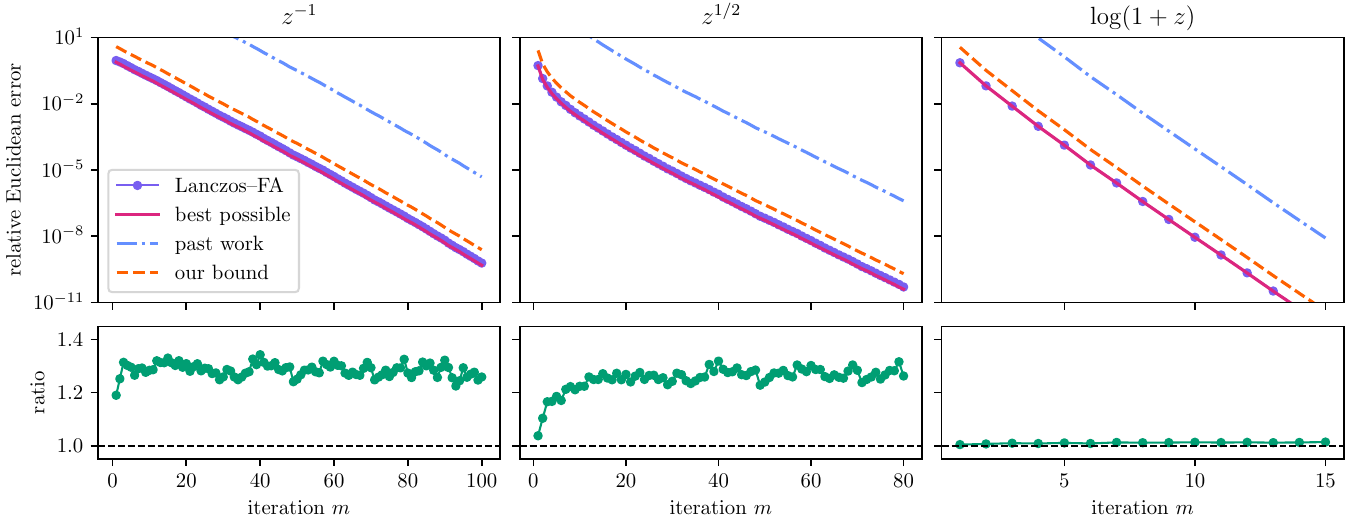}
    \caption{\emph{Top:} Convergence of Lanczos-FA, best approximation from $\Kry_m$, our bound \cref{thm:main_intro}, and \cite[Thm. 3.1]{schweitzer_25} for several matrix functions, on a matrix with $n=400$ eigenvalues geometrically spaced from $1$ to $100$ and Gaussian $\vec{b}$.
    \emph{Bottom:} Ratio of Lanczos-FA error to optimal error. 
    }
    \label{fig:intro}
\end{figure}

The Lanczos method for matrix functions (Lanczos-FA) is perhaps the most widely used algorithm for approximating $f(\vec{A})\vec{b}$.
As with other Krylov Subspace Methods (KSMs), Lanczos-FA iteratively produces approximations to $f(\vec{A})\vec{b}$ from the Krylov subspace
\[
\Kry_m 
= \Kry_m (\vec{A},\vec{b})
:= \operatorname{span}\{\vec{b},\vec{A}\vec{b},\dots,\vec{A}^{m-1}\vec{b}\} = \{p(\vec{A})\vec{b} : p \in \Pol_{m-1}\},
\]
where $\Pol_{m-1}$ is the space of polynomials of degree less than $m$.\footnote{In this work we focus on \emph{polynomial} KSMs. An alternative to these are \emph{rational} KSMs, where the family $\Pol_{m-1}$ is replaced with a family of rational functions \cite{guttelphd}.}
Specifically, the $m$-th Lanczos-FA approximation is\footnote{We remark that \cref{eq:Lanczos-def} is a \emph{mathematical} definition of Lanczos-FA.
A practical implementation will make use of the Lanczos algorithm, which efficiently produces a particular orthonormal basis for $\Kry_m$; see \cite{chen_24} for details.}
\begin{equation}\label{eq:Lanczos-def}
    \vec{f}_m := \vec{V}_m f(\vec{T}_m)\vec{V}_m^*\vec{b},
\end{equation}
where $\vec{V}_m$ is any orthonormal basis of $\Kry_m$ and $\vec{T}_m := \vec{V}_m^* \vec{A}\vec{V}_m$.

This simple algorithm has been highly successful, so over the past 30+ years, understanding the convergence of Lanczos-FA has been the focus of considerable attention in numerical linear algebra \cite{druskin_knizhnerman_91,saad_92,druskin_greenbaum_knizhnerman_98,frommer_simoncini_08,chen_greenbaum_musco_musco_22,schweitzer_25},
theoretical computer science \cite{musco_musco_sidford_18,orecchia_sachdeva_vishnoi_12,sachdeva_vishnoi_14},
and data science/machine learning \cite{ubaru_chen_saad_17,pleiss_jankowiak_eriksson_damle_gardner_20,amsel_chen_greenbaum_musco_musco_24,akhtar_elgadarri_farias_jozefiak_26}. Perhaps the most well-known bound for Lanczos-FA is 
\begin{equation}
\label{eqn:unif_bound}
\| f(\vec{A})\vec{b} - \vec{f}_m \|_2 \leq 2 \|\vec{b}\|_2 \min_{p\in\Pi_{m-1}} \bigg(\max_{z\in[\lmin,\lmax]} |f(z) - p(z) | \bigg),
\end{equation}
which asserts that the convergence of Lanczos-FA is controlled by the best polynomial approximation to $f$ on the interval containing the spectrum of $\vec{A}$.
However, this bound does not take into account any information about $\vec{A}$ besides the smallest and largest eigenvalue, and is therefore extremely pessimistic in many cases. A more refined approach to understanding how well Lanczos-FA works is to compare its accuracy to that of the \emph{best possible} approximation to $f(\vec{A})\vec{b}$ that could be obtained from the Krylov subspace $\Kry_m$ \cite{druskin_greenbaum_knizhnerman_98,amsel_chen_greenbaum_musco_musco_24,schweitzer_25}.
Remarkably the accuracy of Lanczos-FA is nearly optimal; \cite{amsel_chen_greenbaum_musco_musco_24} (NeurIPS '24 Spotlight) provides numerical evidence for this phenomenon, and notes the following:
\begin{quote}
For a wide variety of problem instances, Lanczos-FA is observed to converge almost as quickly as the best approximation to $f(\vec{A})\vec{b}$ that could be returned by any KSM run for the same number of iterations.
\end{quote}
The assertion that Lanczos-FA is nearly optimal means that
\[
 \vvvert f(\vec{A})\vec{b}-\vec{f}_m\vvvert \leq (\text{some moderate value})\cdot  \min_{\vec{x}\in\Kry_m}\vvvert f(\vec{A})\vec{b}-\vec{x}\vvvert,
\]
where $\vvvert\cdot\vvvert$ denotes the norm with which we wish to measure the error. 
A rigorous understanding of this apparent near-instance optimality is of theoretical interest due, if anything, to the mathematical beauty.
There are practical benefits as well, though.
First, if it is known that Lanczos-FA is nearly optimal over $\Kry_m$, it provides insight into the scale of improvements that may be obtained by designing new Krylov subspace methods. 
Second, the error of the best Krylov approximation can be bounded in terms of polynomial approximation on the spectrum of $\vec{A}$. 
For example,
\begin{align}
\min_{\vec{x}\in\Kry_m}\| f(\vec{A})\vec{b}-\vec{x}\|_2
&= \min_{p\in\Pi_{m-1}}\| f(\vec{A})\vec{b}-p(\vec{A})\vec{b}\|_2
\nonumber\\&\leq \min_{p\in\Pi_{m-1}}\| f(\vec{A})-p(\vec{A})\|_2 \|\vec{b}\|_2
\nonumber\\&= \min_{p\in\Pi_{m-1}} \max_{i=1,\ldots,n} | f(\lambda_i)-p(\lambda_i)| \|\vec{b}\|_2. \label{eqn:spec_bound}
\end{align}
This bound depends on fine-grained properties of the spectrum of $\vec{A}$ such as clustered our outlying eigenvalues and is typically fairly sharp; see \cite{chen_greenbaum_musco_musco_22,amsel_chen_greenbaum_musco_musco_24} for further discussion.

\subsection{Near-optimality of Lanczos-FA}

Let $(f,\vec{A},\vec{b})$ define a problem instance, and let $\vvvert\cdot\vvvert$ denote the norm with which we wish to measure the error. To quantify the near-optimality of Lanczos-FA, we define the optimality ratio\footnote{Here we use the convention $1/0=+\infty$ and $0/0 = 1$.}
\begin{equation}\label{eq:constant}
C(f,\vec{A},\vec{b},m,\vvvert\cdot\vvvert)
:= \frac{\vvvert f(\vec{A})\vec{b}-\vec{f}_m\vvvert}{\min_{\vec{x}\in\Kry_m}\vvvert f(\vec{A})\vec{b}-\vec{x}\vvvert }.
\end{equation}
If $C(f,\vec{A},\vec{b},m,\vvvert\cdot\vvvert)$ is not too large, then the convergence of Lanczos-FA is close to that of the best possible Krylov approximation; i.e. near-optimal. 

In the case that $f(z) = 1/z$ and $\vec{A}$ is positive-definite, then Lanczos-FA is mathematically equivalent to the celebrated Conjugate Gradient algorithm, and hence optimal in the $\vec{A}$-norm \cite{greenbaum_97}.\footnote{For positive definite $\vec{M}$, $\|\vec{x}\|_{\vec{M}} := \sqrt{\vec{x}^*\vec{M}\vec{x}}$.}
That is, 
\begin{equation}
    \label{eqn:cg_Aopt}
    C(1/z, \vec{A},\vec{b}, m,\|\cdot\|_{\vec{A}}) = 1.
\end{equation}
Guarantees for other norms (e.g. the Euclidean norm) can then be obtained by the equivalence of norms; see e.g. \cite[(2.2)]{schweitzer_25}. Several past works have proved weaker versions of near-optimality for other functions such as the matrix exponential \cite{druskin_greenbaum_knizhnerman_98} and a class of rational functions \cite{amsel_chen_greenbaum_musco_musco_24}.
However, these bounds are not particularly satisfying in that they do not prove true instance optimality, and at least in the case of \cite{amsel_chen_greenbaum_musco_musco_24}, have an exponential dependence on the condition number $\kappa(\vec{A}) := \lmax/\lmin$ of $\vec{A}$; see \cite{schweitzer_25} for a discussion.

A natural generalization of the inverse function is the class of \emph{Stieltjes functions}.
Indeed, a function $f$ is a Stieltjes function if
\begin{equation}\label{eq:stieltjes}
f(z) = \int_0^\infty \frac{1}{z+t}\dmu,
\end{equation}
where $\mu$ is a positive measure on $[0,\infty)$ with $\int_0^\infty (1+t)^{-1}\dmu < \infty$.
More broadly, define the family
\[
\mathcal{S} := \{ f : f(z)~\text{or}~f(z)/z~\text{is Stieltjes}\}.
\]
This class contains many scientifically relevant functions including the inverse $z^{-1}$, fractional powers $z^{-\gamma}$ with $\gamma \in (-1,1)$, every rational function $\sum_i \sigma_i(z+t_i)^{-1}$ with $\sigma_i > 0$ and $t_i \ge 0$, and the logarithm $\log(1+z)$. Recently, in a major advance, Schweitzer proved Lanczos-FA satisfies a near-instance optimality bound for $\mathcal{S}$ in the Euclidean norm:
\begin{theorem*}[{\cite[Theorems 3.1 and 4.1]{schweitzer_25}}]
For any $f\in\mathcal{S}$, Hermitian positive-definite $\vec{A}\in\C^{n\times n}$, $\vec{b}\in\C^n$, and $m\geq 1$,
\[
C(f,\vec{A},\vec{b},m,\|\cdot\|_2) \leq 1 + \kappa(\vec{A})^2.
\]
\end{theorem*}
\subsection{Our results}

Our first result, which we prove in \cref{sec:upper}, is a strengthening and generalization of the above result of \cite{schweitzer_25}.\footnote{In \cite{schweitzer_25}, a stronger bound based on an intermediate quantity produced by the Lanczos algorithm is also proved.}
\begin{theorem}\label{thm:main_intro}
For any $f\in\mathcal{S}$, Hermitian positive-definite $\vec{A}\in\C^{n\times n}$, $\vec{b}\in\C^n$, $m\geq 1$, and $\alpha\in\R$,
\[
C(f,\vec{A},\vec{b},m,\|\cdot\|_{\vec{A}^{\alpha}}) \leq \frac{1}{2}\big( \kappa(\vec{A})^{E/2} + \kappa(\vec{A})^{-E/2} \big),
\qquad
E:=\max\{\alpha,1-\alpha\}.
\]
\end{theorem}
\noindent In particular, when $\alpha=0$, we obtain a near-optimality guarantee for the Euclidean norm that scales with $\kappa(\vec{A})^{1/2}$.

\Cref{thm:main_intro}, as with past work \cite{amsel_chen_greenbaum_musco_musco_24,schweitzer_25}, depends on the condition number $\kappa(\vec{A})$ of $\vec{A}$.
One wonders whether this dependence can be removed, for instance, by working in a different norm.
Our second result, which we prove in \cref{sec:lower}, asserts that the value of the upper bound in \cref{thm:main_intro} cannot be improved. 

\begin{theorem}\label{thm:lower_m}
Fix $\alpha \in \R$, $E := \max\{\alpha, 1-\alpha\}$, $\kappa > 1$ and $m \geq 1$.
For every $\eta \in (0,1]$, $n \geq m+1$ and every Hermitian positive definite $\vec{A} \in \mathbb{C}^{n \times n}$ with at least $m+1$ distinct eigenvalues and condition number $\kappa$, there exists a non-zero vector $\vec{b}$ and Stieltjes function $f$ such that
\[
C(f,\vec{A},\vec{b},m,\|\cdot\|_{\vec{A}^\alpha}) \ge (1-\eta)\,\frac12\big(\kappa^{E/2} + \kappa^{-E/2}\big).
\]
\end{theorem}

\begin{figure}
    \centering
    \includegraphics[scale=.6]{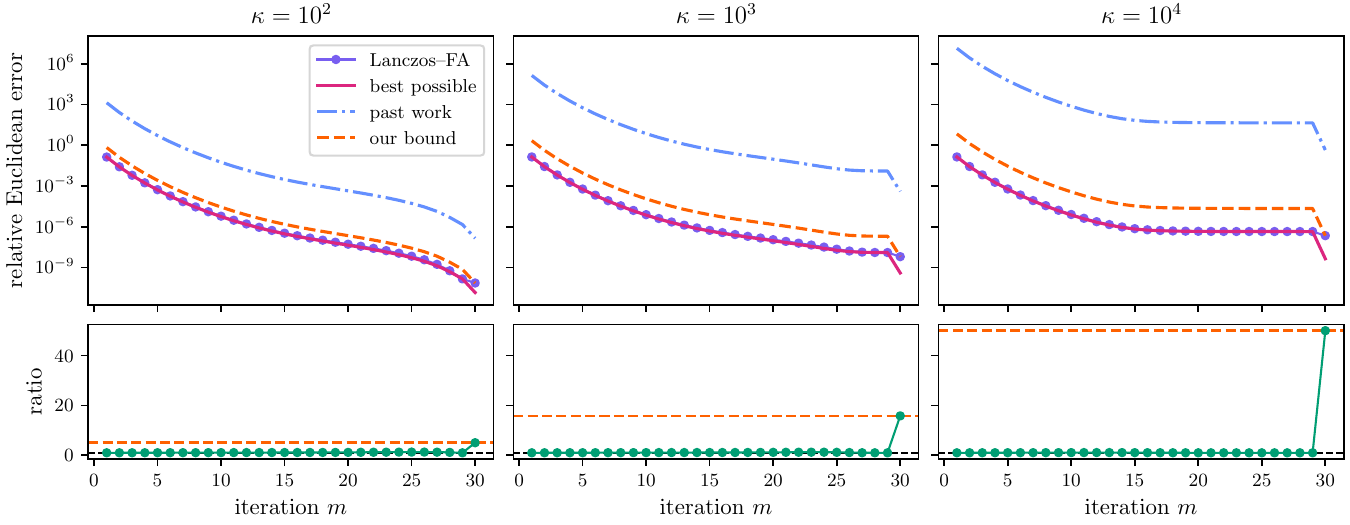}
    \caption{\emph{Top:} Convergence of Lanczos-FA, best approximation from $\Kry_m$, our bound \cref{thm:main_intro}, and \cite[Thm. 3.1]{schweitzer_25} for the hard instance used to prove \cref{thm:lower_m} at $m=30$ and varying $\kappa$, using $n=m+1$ geometrically spaced eigenvalues.
    \emph{Bottom:} Ratio of Lanczos-FA error to optimal error. Observe that at $m=30$ our upper bound (indicated by the upper dashed line) is sharp.}
    \label{fig:hard}
\end{figure}

Since Lanczos-FA is mathematically equivalent to Conjugate Gradient when $f(z) = 1/z$ and $\vec{A}$ is Hermitian positive definite \cite{greenbaum_97}, we also obtain  bounds for Conjugate Gradient.
For instance, we have the following guarantee for the Euclidean norm.
\begin{remark}
    \Cref{thm:main_intro} implies $C(1/z, \vec{A},\vec{b}, m,\|\cdot\|_{2}) \leq \tfrac12(\kappa(\vec{A})^{1/2} + \kappa(\vec{A})^{-1/2})$.
    This improves on the typical bound $C(1/z, \vec{A},\vec{b}, m,\|\cdot\|_{2}) \leq \kappa(\vec{A})^{1/2}$ obtained by using \cref{eqn:cg_Aopt} and that $\lmin(\vec{A})\|\vec{x}\|_2^2 \leq \|\vec{x}\|_{\vec{A}}^2 \leq \lmax(\vec{A})\|\vec{x}\|_2^2$.
    Moreover, the hard instance used in the proof of \cref{thm:lower_m} at $\alpha = 0$ uses the function $1/z$, so the value $\tfrac12(\kappa(\vec{A})^{1/2} + \kappa(\vec{A})^{-1/2})$ cannot cannot be improved.
\end{remark}

\subsection{Comparison of analysis strategies}

The argument of \cite{schweitzer_25} is carried out at the matrix level by decomposing Lanczos-FA error vector $f(\vec{A})\vec{b} - \vec{f}_m$ into part due to the error vector of the best Euclidean approximation $\vec{V}_m\vec{V}_m^\T f(\vec{A})\vec{b}$ and a part orthogonal to the best approximation.
It is shown that the orthogonal component can be expressed in terms of submatrices of the final matrix $\vec{T}_m$ obtained by running Lanczos to completion.
The error is then bounded by a careful application of the Woodbury identity and by analyzing certain auxiliary functions which are proved to be Stieltjes and hence satisfy certain monotonicity properties.

In contrast, our proof of \cref{thm:main_intro} works with the \emph{norm} of the Lanczos-FA error.
Similar to \cite{schweitzer_25}, which we decompose the Lanczos-FA error into a contribution from the optimal error and an excess term.
However, by working with the error norm (rather than the error vector) we can make use of properties of weighted $L^2$ spaces.
This simplifies the required bookkeeping substantially and helps avoid introducing unnecessary slackness into our analysis.
Indeed, our proof of \cref{thm:lower_m} works by identifying when the inequality we use is an equality.
Moreover, our approach works for any norm induced by a positive function of $\vec{A}$.

\subsection{Notation}
\label{sec:notation}

We measure errors in norms induced by functions of $\vec{A}$.
For $\nu$ positive on the spectrum of $\vec{A}$ we define
\[
\langle \vec{x},\vec{y}\rangle_\nu
= \vec{x}^* \nu(\vec{A})\vec{y}
,\qquad\|\vec{x}\|_\nu = \langle\vec{x},\vec{x}\rangle_\nu^{1/2}.
\]
Then $\nu \equiv 1$ corresponds to the Euclidean norm and $\nu(z) = z$ the $\vec{A}$-norm. Let $M$ be the first index for which $\Kry_M = \Kry_{M+1}$.
When $m\geq M$, it is well-known that Lanczos-FA is exact  \cite[\S3.2]{saad_92}.
Therefore, we assume, without loss of generality, that $m < M$.
In this case the Ritz values (eigenvalues of $\vec{T}_m$) are denoted $\theta_1,\dots,\theta_m$ are distinct and are contained in $(\lmin,\lmax)$; see \cite[\S4]{golub_meurant_09}.
Finally, we define the polynomial
\begin{equation}\label{eq:pim}
    \pi_m(z) := \prod_{j=1}^{m}(z-\theta_j),
\end{equation}
which will take a central role in our proof of the upper bound. 

\section{The upper bound}
\label{sec:upper}

This section is dedicated to the proof of \cref{thm:main_intro}. 
Our proof can be broken down into simple steps. 
Let $\vec{P}$ be the $\langle \cdot, \cdot\rangle_{\nu}$ orthogonal projection onto $\Kry_m$. 
The first step is noting that since $\vec{f}_m \in \Kry_m$ an application of Pythagorean theorem yields
\begin{align}
\begin{split}\label{eq:intropyth}
    \|f(\vec{A})\vec{b} - \vec{f}_m\|_{\nu}^2 &= \|(\vec{I} - \vec{P}) f(\vec{A})\vec{b}\|_{\nu}^2 + \|\vec{P}(f(\vec{A})\vec{b} - \vec{f}_m)\|_{\nu}^2\\
    &=\Big(\min_{\vec{x}\in\Kry_m}\|f(\vec{A})\vec{b}-\vec{x}\|_\nu\Big)^2 + \|\vec{P}(f(\vec{A})\vec{b} - \vec{f}_m)\|_{\nu}^2.
\end{split}
\end{align}
Next, we use a well-known fact if $q_{m-1}\in \Pol_{m-1}$ is the polynomial that interpolates $f$ on the Ritz values (eigenvalues of $\vec{T}_m$), then
\begin{equation}\label{eq:interpolation}
    \vec{f}_m = q_{m-1}(\vec{A})\vec{b};
\end{equation}
see \cite[Corollary 6.3]{chen_24}.
Existing results from \cite{frommer_guttel_schweitzer_14} allow us to obtain an explicit representation for the error function $e = f-q_{m-1}$. 
This expression will allow us to show $\|\vec{P}(f(\vec{A})\vec{b} - \vec{f}_m)\|_{\nu} \leq \tau \|f(\vec{A})\vec{b} - \vec{f}_m\|_{\nu}$ for some $\tau \in (0,1)$ depending only on the eigenvalues of $\vec{A}$, the choice of norm $\|\cdot\|_{\nu}$, and the function $f$. 
Combining this inequality and \cref{eq:intropyth} yields
\begin{equation}\label{eq:desiderata}
    \|f(\vec{A})\vec{b} - \vec{f}_m\|_{\nu}\leq \frac{1}{\sqrt{1-\tau^2}} \min_{\vec{x}\in\Kry_m}\|f(\vec{A})\vec{b}-\vec{x}\|_\nu.
\end{equation}
So, it suffices to obtain a sufficiently good value of $\tau$.

We begin by recalling an existing explicit representation for the error function $e = f-q_{m-1}$.

\begin{lemma}\label[lemma]{lem:factorisation}
Suppose $1\leq m<M$, and let $q_{m-1}$ be the polynomial interpolating $f$ at the Ritz values $\theta_1,\ldots,\theta_m$. Define $e = f - q_{m-1}$. Then
\begin{enumerate}
    \item 
If $f$ is Stieltjes with representing measure $\mu$, then $e = (-1)^{m}\pi_m\varsigma_m$
where
\[
\varsigma_m(z)
=
\int_0^\infty\frac{1}{z+t}\,\mathrm{d}\varrho_m(t),
\qquad
\mathrm{d}\varrho_m(t)
=
\frac{\mathrm{d}\mu(t)}
     {\prod_{j=1}^m(\theta_j+t)}.
\]

\item 
If $f(z)/z$ is Stieltjes with representing measure $\mu$, then $e = (-1)^{m+1}\pi_m\varsigma_m$
where
\[
\varsigma_m(z)
=
\int_0^\infty\frac{1}{z+t}\,\mathrm{d}\varrho_m(t),
\qquad
\mathrm{d}\varrho_m(t)
=
\frac{t\,\mathrm{d}\mu(t)}
     {\prod_{j=1}^m(\theta_j+t)}.
\]
\end{enumerate}
In either case $\varsigma_m\geq0$ on $(0,\infty)$. If
$\varrho_m\neq0$, then $\varsigma_m>0$ there.
\end{lemma}

\begin{proof}
The result for Stieltjes $f$ is \cite[(3.9)]{frommer_guttel_schweitzer_14}. Note that the authors of \cite{frommer_guttel_schweitzer_14} equivalently define Stieltjes functions as $f(z) = \int_{-\infty}^0 \frac{1}{t-z} \dmu$. The extra $(-1)^m$ comes from a change of variables $t \mapsto -t$ to convert their definition to \cref{eq:stieltjes}.

Now, suppose $f(z)/z$ is a Stieltjes function with representing measure $\mu$. We proceed similarly to \cite{frommer_guttel_schweitzer_14}.
Fix $t \geq 0$ and let $q^{(t)} \in \Pol_{m-1}$ interpolate $z \mapsto z/(z+t)$ at $\theta_1, \ldots, \theta_m$.
The numerator of the rational function $z/(z+t) - q^{(t)}(z) = (z - (z+t)q^{(t)}(z))/(z+t)$ has degree at most $m$ and vanishes at the simple nodes $\theta_1, \ldots, \theta_m$, so it equals $c_t \pi_m(z)$; evaluating the numerator at $z = -t$ yields $c_t \pi_m(-t) = -t$, and hence
\begin{equation*}
e^{(t)}(z):= \frac{z}{z+t} - q^{(t)}(z) = (-1)^{m+1} \frac{\pi_m(z)}{z+t} \frac{t}{\prod_j (\theta_j + t)}.
\end{equation*}
Integrating we see
\[
f - q_{m-1} = \int_0^\infty e^{(t)}(z) \dmu = (-1)^{m+1} \pi_m \varsigma_m,
\]
where we have used that $\int q^{(t)}(z)\dmu \in \Pi_{m-1}$ and that the above integral vanishes at $\theta_1, \ldots, \theta_m$.
Since $\theta_i > \lmin$, $t /((z+t)\prod_j (\theta_j+t)) \leq C /(z+t)$, so the integral exists and is positive for any $z\geq 0$.
\end{proof}

Our proof will also make repeated use of the fact that, since $m < M$, we have
\begin{equation}\label{eq:orthogonality}
    \vec{b}^* \pi_m(\vec{A})\, p(\vec{A})\vec{b} = 0 \text{  for all } p \in \Pol_{m-1},
\end{equation}
where $\pi_m$ is as in \cref{eq:pim};
see \cite[\S2,\S4]{golub_meurant_09}.\footnote{This is usually stated in terms of inner products with respect to the spectral measure $\mu_{\vec{b}} = \sum_{i=1}^n |\vec{w}_i^* \vec{b}|^2 \delta_{\lambda_i}$, where $\delta_{\lambda}$ denotes the Dirac mass centered at $\lambda$. Noting that $\int p g \mathrm{d} \mu_{\vec{b}} = \vec{b}^* p(\vec{A}) g(\vec{A}) \vec{b}$ and and that $\pi_m$ is the degree-$m$ monic orthogonal polynomial of $\mu_{\vec{b}}$, and hence orthogonal to all lower-degree polynomials, yields \cref{eq:orthogonality}.\label{footnote:measure}} We also note that $\pi_m$ is the unique monic polynomial of degree at most $m$ satisfying \eqref{eq:orthogonality}.

With these results at hand, we are ready state the following theorem, which will lead to \Cref{thm:main_intro} as a corollary.

\begin{theorem}\label{thm:upper}
Let $\vec{A}$ be Hermitian positive definite, let $f\in\mathcal{S}$, and let $\nu$ be positive on the spectrum of $\vec{A}$. 
Put
\[
R := \frac{\max_{i=1,\ldots, n}\nu(\lambda_i)\varsigma_m(\lambda_i)}{\min_{i=1,\ldots, n}\nu(\lambda_i)\varsigma_m(\lambda_i)} \ \ge 1
\]
with $\varsigma_m$ as in \cref{lem:factorisation}.
Then with $\vec{f}_m$ as in \cref{eq:Lanczos-def} we have
\begin{equation}\label{eq:upper}
\|f(\vec{A})\vec{b}-\vec{f}_m\|_\nu \leq  \frac12\left(\sqrt{R}+\frac{1}{\sqrt R}\right)\min_{\vec{x}\in\Kry_m}\|f(\vec{A})\vec{b}-\vec{x}\|_\nu .
\end{equation}
\end{theorem}
\begin{proof}
    Define $\vec{e}:=f(\vec{A})\vec{b} - \vec{f}_m = (f(\vec{A}) - q_{m-1}(\vec{A})) \vec{b} = e(\vec{A})\vec{b}$, where $q_{m-1}$ is the interpolating polynomial \cref{eq:interpolation} and $e = f-q_{m-1}$ the error function. 
    We will begin with providing an upper bound $\tau \in [0,1)$ for $\|\vec{P}\vec{e}\|_{\nu}/\|\vec{e}\|_{\nu}$, where $\vec{P}$ is the $\langle \cdot, \cdot\rangle_{\nu}$ orthogonal projection onto $\Kry_m$.
    By the Pythagorean theorem \cref{eq:intropyth}, this allows us to obtain \cref{eq:desiderata}.

    Let $g := \nu\varsigma_m > 0$, let $p \in \Pol_{m-1}$ with $\|p(\vec{A})\vec{b}\|_\nu=1$ and let $\hat g > 0$ be any constant.
Using \cref{lem:factorisation} and then that $\pi_m$ is orthogonal to $\hat{g}p\in\Pi_{m-1}$ (i.e. \cref{eq:orthogonality}),
\[
\left|\ip{\vec{e}}{p(\vec{A})\vec{b}}_\nu\right|
= \left| \vec{b}^* \pi_m(\vec{A})\varsigma_m(\vec{A})\nu(\vec{A}) p(\vec{A})\vec{b}\right|
= \left| \vec{b}^* \pi_m(\vec{A})\big(g(\vec{A})-\hat g\vec{I}\big) p(\vec{A})\vec{b}\right|.
\]
Factoring the last expression as
\[
\vec{b}^* \pi_m(\vec{A})\big(g(\vec{A})-\hat g\vec{I}\big) p(\vec{A})\vec{b}
= \big(\pi_m(\vec{A})\varsigma_m(\vec{A})\nu(\vec{A})^{1/2}\vec{b}\big)^* \big(\vec{I}-\hat g\, g(\vec{A})^{-1}\big)\nu(\vec{A})^{1/2} p(\vec{A})\vec{b}
\]
and applying the Cauchy--Schwarz inequality in the Euclidean inner product gives
\begin{align*}
\big|\ip{\vec{e}}{p(\vec{A})\vec{b}}_\nu\big|
&\leq  \big\|\pi_m(\vec{A})\varsigma_m(\vec{A})\vec{b}\big\|_\nu \cdot \big\| \big(\vec{I} - \hat g\, g(\vec{A})^{-1}\big) \nu(\vec{A})^{1/2} p(\vec{A})\vec{b} \big\|_2
\\&\leq  \|\vec{e}\|_\nu \cdot \big\|\vec{I} - \hat g\, g(\vec{A})^{-1}\big\|_2 \, \big\|p(\vec{A})\vec{b}\big\|_\nu
\\&=  \|\vec{e}\|_\nu\cdot\max_{i=1,\ldots, n}\Big|1-\frac{\hat g}{g(\lambda_i)}\Big|,
\end{align*}
where we have used that $\vec{e} = \pm\pi_m(\vec{A})\varsigma_m(\vec{A})\vec{b}$, that $\|\nu(\vec{A})^{1/2}\vec{x}\|_2 = \|\vec{x}\|_\nu$, and that $\|p(\vec{A})\vec{b}\|_\nu = 1$.
Now, since $\vec{P}\vec{e}\in\Kry_m$ and every element of $\Kry_m$ has the form $p(\vec{A})\vec{b}$ with $p\in\Pi_{m-1}$,
\begin{equation*}
    \| \vec{P}\vec{e} \|_\nu
= \sup_{\substack{p\in\Pi_{m-1}\\\|p(\vec{A})\vec{b}\|_\nu=1}} \big|\ip{\vec{e}}{p(\vec{A})\vec{b}}_\nu\big|
\leq 
\|\vec{e}\|_\nu\cdot\max_{i=1,\ldots, n}\Big|1-\frac{\hat g}{g(\lambda_i)}\Big|,
\end{equation*}
Let $g_{\min} := \min\limits_{i=1,\ldots,n} g(\lambda_i)$ and $g_{\max}:=\max\limits_{i=1,\ldots,n} g(\lambda_i)$. Choosing $\hat g = 2g_{\min} g_{\max}/(g_{\min}+g_{\max})$ to be, the harmonic mean of the extreme values $g_{\min}, g_{\max}$ of $g$ on the spectrum of $\vec{A}$, yields
\begin{equation*}
\tau:=\max_{i=1,\ldots, n}\Big|1-\frac{\hat g}{g(\lambda_i)}\Big|
= \max_{i=1,\ldots, n}\Big|1-\frac{2g_{\min}g_{\max}}{(g_{\min}+g_{\max})g(\lambda_i)}\Big|
= \frac{g_{\max}-g_{\min}}{g_{\max}+g_{\min}}
= \frac{R-1}{R+1},
\end{equation*}
from which we obtain the bound 
\[
\frac{1}{\sqrt{1-\tau^2}} =  \frac{R+1}{2\sqrt R} = \frac12\left(\sqrt R + \frac{1}{\sqrt R}\right). 
\]
This gives the result.
\end{proof}

In the following, we will show that $R \leq \kappa^{\max\{\alpha,1-\alpha\}}$. 
For this, we need the following lemma. 

\begin{lemma}\label[lemma]{lem:oscillation}
Let $\varsigma(z) = \int_0^\infty (z+t)^{-1}\mathrm{d}\varrho(t)$ with $\varrho$ positive, and let $\alpha \in \R$.
The function $g_\alpha(z) := z^{\alpha}\varsigma(z)$ satisfies
\begin{equation*}
\frac{\max_{i=1,\ldots, n} g_\alpha(\lambda_i)}{\min_{i=1,\ldots, n} g_\alpha(\lambda_i)} \leq \kappa^{E}, 
\qquad
E:=\max\{\alpha,1-\alpha\}.
\end{equation*}
\end{lemma}
\begin{proof}
By positivity of $\varrho$, that $z \mapsto z/(z+t)$ is non-decreasing, and $\varsigma$ is non-increasing we immediately obtain $\varsigma(z_2) \leq \varsigma(z_1)$ and $z_1\varsigma(z_1) \leq z_2 \varsigma(z_2)$.
Hence, for any $\lmin \le z_1 \le z_2 \le \lmax$ we have
\begin{equation*}
    \frac{z_1^{\alpha}}{z_2^{\alpha}} \leq \frac{z_1^{\alpha}}{z_2^{\alpha}} \cdot \frac{\varsigma(z_1)}{\varsigma(z_2)} = \frac{g_{\alpha}(z_1)}{g_{\alpha}(z_2)} \leq \frac{z_1^{\alpha-1}}{z_2^{\alpha-1}}.
\end{equation*}
Using $(z_2/z_1)^{\max\{\alpha,1-\alpha\}} \leq \kappa^{\max\{\alpha,1-\alpha\}}$, any two values of $g_\alpha$ on $[\lmin,\lmax]$ differ by at most a factor $\kappa^{\max\{\alpha,1-\alpha\}}$, which is the result.
\end{proof}

With \Cref{thm:upper} and \Cref{lem:oscillation} at hand, we are ready to prove \Cref{thm:main_intro}.

\begin{proof}[Proof of \cref{thm:main_intro}]
Combine \cref{thm:upper} with \cref{lem:oscillation} and the monotonicity of $R \mapsto \tfrac12(\sqrt R + R^{-1/2})$ on $[1,\infty)$.
\end{proof}

\section{The lower bound}
\label{sec:lower}

We now turn our attention to \cref{thm:lower_m}.
Note that our proof of \cref{thm:main_intro} only has two sources of slackness: \cref{lem:oscillation} and the use of Cauchy--Schwarz in the proof of \cref{thm:upper}.
Our lower bound constructs a ``hard instance'' for which both inequalities are simultaneously equality.
In particular, by using the function $f_s(z) = 1/(z+s)$, for some $s\geq 0$ tuned to $\alpha$, we ensure that \cref{lem:oscillation} holds with equality, and by carefully choosing $\vec{b}$ we ensure that \cref{thm:upper} holds with equality.

By choosing $\vec{b}$ to have nonzero inner product with the eigenvectors corresponding to the largest and smallest eigenvalues, as well as with $m-1$ other eigenvectors, and to be orthogonal to all the remaining eigenvectors, it suffices to consider the case $n = m+1$.
Fix a Hermitian positive definite matrix $\vec{A} \in \mathbb{C}^{m+1 \times m+1}$ with simple eigenvalues, a vector $\vec{b} \in \mathbb{C}^{m+1}$, and a Stieltjes function $f_s(z) = 1/(z+s)$, for some $s \geq 0$. 
The proof consists of three steps. In the first two steps we derive \emph{exact} expressions for the $m$-th Lanczos-FA error $\|f_s(\vec{A})\vec{b} - \vec{f}_m\|_{\vec{A}^{\alpha}}$ and the optimal approximation error $\min_{\vec{x}\in\Kry_m}\|f_s(\vec{A})\vec{b}-\vec{x}\|_{\vec{A}^\alpha}$. 
In the third step we show that their ratio can be written as
\begin{equation}\label{eq:cauchy}
    \frac{\|f_s(\vec{A})\vec{b}-\vec{f}_m\|_{\vec{A}^\alpha}}{\min_{\vec{x}\in\Kry_m}\|f_s(\vec{A})\vec{b}-\vec{x}\|_{\vec{A}^\alpha}} = \frac{\|\vec{v}\|_2 \|\vec{u}\|_2}{|\vec{v}^* \vec{u}|},
\end{equation}
where $\vec{v} = \vec{v}(\vec{A},\vec{b},s)$ and $\vec{u} = \vec{u}(\vec{A},\vec{b},s)$ are vectors of length $m+1$ that depend on $\vec{A}, \vec{b},$ and $s$. 
We will then tune $\vec{b}$ and $s$ with respect to $\vec{A}$ so that \cref{thm:lower_m} holds.
The parameter $s$ is only needed when $\alpha > 1/2$. The reader interested in the Euclidean norm ($\alpha = 0$) may assume that $s = 0$ and the only tunable choice is the vector $\vec{b}$. 

Let $\vec{A} = \sum_{i=1}^{m+1}\lambda_i \vec{w}_i \vec{w}_i^*$ be its eigendecomposition with simple eigenvalues $\lambda_1 > \cdots > \lambda_{m+1}$, eigenvectors $\vec{w}_1,\ldots,\vec{w}_{m+1}$, and condition number $\kappa = \lambda_1/\lambda_{m+1}$. Define $\omega_i := \prod_{k \neq i}(\lambda_i - \lambda_k)$ and the $m$-th divided difference for a function $f$:
\begin{equation*}
    f[\lambda_1, \dots, \lambda_{m+1}] := \sum_{i=1}^{m+1} \frac{f(\lambda_i)}{\omega_i}.
\end{equation*}
We will use the fact that for any polynomial $p \in \Pi_{m}$, $p[\lambda_1,\dots,\lambda_{m+1}]$ is the coefficient on the degree-$m$ term of $p$'s monomial expansion \cite[(1.10)]{higham_08}. To derive an expression for the $m$-th Lanczos-FA error $\|f_s(\vec{A})\vec{b} - \vec{f}_m\|_{\vec{A}^{\alpha}}$ we first note that \emph{any} vector $\vec{b}$ satisfying $\vec{u}_i^*\vec{b} \neq0$ for all $i$ can be written as
\begin{equation}\label{eq:b}
    \vec{b} = \sum\limits_{i=1}^{m+1} \widehat{b}_i\vec{w}_i , \quad |\widehat{b}_i|^2 := \frac{1}{g_i \omega_i^2} > 0,
\end{equation}
for some $g_1,\ldots,g_{m+1} > 0$. 
We then observe that \cref{lem:factorisation} and \cref{eq:interpolation} allow us to compute $\|f_s(\vec{A})\vec{b} - \vec{f}_m\|_{\vec{A}^{\alpha}}$ if we know the values of $\pi_m$ on the eigenvalues of $\vec{A}$. 
The reason for expressing the vector $\vec{b}$ as in \cref{eq:b} is because this allows us to express $\pi_m(\lambda_i)$ in terms of $g_i$ and $\omega_i$, as shown in the following lemma.

\begin{remark}
    Our proof is similar in spirit to the approach used by Greenbaum in \cite{greenbaum_79} to prove lower bounds for Conjugate Gradient in the $\vec{A}$-norm; see also \cite[Appendix C]{amsel_chen_greenbaum_musco_musco_24}.
\end{remark}

\begin{lemma}\label[lemma]{lem:prescribe}
For a vector $\vec{b} \in \mathbb{C}^{m + 1}$ and $g_1,\ldots,g_{m+1}$ as defined in \cref{eq:b}, define $G :=  \sum_i g_i$. Then $\dim(\Kry_m) = m$ and the Ritz polynomial $\pi_m(z) = \prod_i (z-\theta_i)$ of $\vec{b}$ at step $m$ satisfies $\pi_m(\lambda_i) = g_i\omega_i/G \neq0$. 
\end{lemma}

\begin{proof}
The vector $\vec{b}$ has non-zero inner product with all eigenvectors of $\vec{A}$. Hence, the grade of $\vec{b}$ is $m+1$ and by \cite[Proposition 6.2]{saaditerative} $\dim(\Kry_m) = m$. 
The function $\pi$ defined by $\pi(\lambda_i) := g_i\omega_i/G$ extends uniquely to a polynomial of degree at most $m$; the coefficient on the degree-$m$ term is $\pi[\lambda_1,\dots,\lambda_{m+1}] = \sum_i g_i/G = 1$, so $\pi$ is monic of degree exactly $m$.
Moreover, for every $p \in \Pol_{m-1}$,
\[
\sum_{i=1}^{m+1} |\widehat{b}_i|^2\, \pi(\lambda_i)\, p(\lambda_i)= \frac{1}{G} \sum_{i=1}^{m+1} \frac{p(\lambda_i)}{\omega_i} = \frac{1}{G} p[\lambda_1,\dots,\lambda_{m+1}] = 0.
\]
Thus, $\pi$ is the monic degree-$m$ orthogonal polynomial of $\mu_{\vec{b}}$, i.e., $\pi = \pi_m$ (see \cref{eq:orthogonality} and \cref{footnote:measure}). 
\end{proof}
In the second step of the proof, we need to derive an expression for the optimal approximation error $\min_{\vec{x}\in\Kry_m}\|f_s(\vec{A})\vec{b}-\vec{x}\|_{\vec{A}^\alpha}$. 
In particular, we will prove that this error can be expressed as a divided difference of the function $e$ from \cref{lem:factorisation}. 
The following lemma will be essential for that purpose.

\begin{lemma}\label[lemma]{lem:divdiff}
Let $u_1, \dots, u_{m+1} > 0$ be weights on the nodes $\lambda_1,\ldots, \lambda_{m+1}$ and let $e$ be any function on the nodes.
Then
\[
\min_{p\in\Pol_{m-1}} \sum_{i=1}^{m+1} u_i \big|e(\lambda_i)-p(\lambda_i)\big|^2 = \frac{|e[\lambda_1,\dots,\lambda_{m+1}]|^2}{\sum_{i=1}^{m+1} (u_i\omega_i^2)^{-1}} .
\]
\end{lemma}

\begin{proof}
Functions on $m+1$ nodes form an $(m+1)$-dimensional space in which $\Pol_{m-1}$ has dimension $m$, so the orthogonal complement of $\Pol_{m-1}$ in the $u$-weighted inner product is spanned by a single function $h$.
The function $h(\lambda_i) := (u_i\omega_i)^{-1}$ lies in this complement, since $\sum_i u_i h(\lambda_i) p(\lambda_i) = p[\lambda_1,\dots,\lambda_{m+1}] = 0$ for all $p \in \Pol_{m-1}$.
The minimum error is the squared projection onto $h$, namely $|\langle e, h\rangle_u|^2/\|h\|_u^2$, where $\langle e, h\rangle_u = \sum_i e(\lambda_i)/\omega_i = e[\lambda_1,\dots,\lambda_{m+1}]$ while $\|h\|_u^2 = \sum_i (u_i\omega_i^2)^{-1}$.
\end{proof}
With \cref{lem:prescribe} and \cref{lem:divdiff} at hand, we are ready to prove \cref{thm:lower_m}.

\begin{proof}[Proof of \cref{thm:lower_m}]
 As argued above, it suffices to consider the case when $n=m+1$. 
 Let $s \geq 0$, to be chosen at the end, and let $ f_s(z) := (z+s)^{-1} \in \mathcal{S}$.
Write $\nu_i := \lambda_i^\alpha$ and $\varphi_{i,s} := \lambda_i^\alpha/(\lambda_i+s)$.
Given $h_1, \dots, h_{m+1} > 0$ to be chosen at the end, set $g_i := h_i\nu_i$ and let $\vec{b}$ be as in \cref{eq:b}.

\emph{Step 1: An expression for the Lanczos-FA error.} We begin with the Lanczos-FA error.
By \cref{eq:interpolation} and \cref{lem:factorisation} with $\mu = \delta_s$ (the Dirac point mass at $s$), the error function is
\[
e(z) = f_s(z) - q_{m-1}(z) = c\, \pi_m(z)(z+s)^{-1}
,\qquad 
c := \frac{(-1)^m}{\prod_j (\theta_j+s)}.
\]
Now, by \cref{lem:prescribe} and the fact $G = \sum_j g_j = \sum_j h_j \nu_j$,
\[
e(\lambda_i) 
= c \frac{\pi_m(\lambda_i)}{\lambda_i+s}
= \frac{c}{\sum_j h_j\nu_j} \frac{g_i\omega_i}{\lambda_i+s}
= \frac{c}{\sum_j h_j\nu_j} \frac{h_i \nu_i \omega_i}{\lambda_i+s}
= \frac{c}{\sum_j h_j\nu_j} h_i \omega_i \varphi_{i,s}.
\]
If
\[
u_i := \nu_i |\widehat{b}_i|^2 
= \nu_i (g_i \omega_i^2)^{-1}
= (h_i\omega_i^2)^{-1},
\]
then
\[
\|f_s(\vec{A})\vec{b}-\vec{f}_m\|_{\vec{A}^\alpha}^2 
= \sum_{i=1}^{m+1} u_i e(\lambda_i)^2 
= \frac{c^2}{\big(\sum_j h_j\nu_j\big)^2}  \sum_{i=1}^{m+1}  h_i\,\varphi_{i,s}^2,
\]
which yields an exact expression for the Lanczos-FA error.

\emph{Step 2: An expression for the optimal error.} We now turn our attention to the optimal error.
First, 
\begin{equation*}
e[\lambda_1,\dots,\lambda_{m+1}] 
= \sum_{i=1}^{m+1} \frac{e(\lambda_i)}{\omega_i} 
= \frac{c}{\sum_j h_j\nu_j} \sum_{i=1}^{m+1} h_i\, \varphi_{i,s}, \quad 
\sum_{i=1}^{m+1} (u_i\omega_i^2)^{-1} 
= \sum_{i=1}^{m+1} h_i.
\end{equation*}
Now, since $q_{m-1}\in\Pi_{m-1}$,
\[
\min_{\vec{x}\in\Kry_m}\|f_s(\vec{A})\vec{b}-\vec{x}\|_{\vec{A}^\alpha}^2
= \min_{p\in\Pi_{m-1}} \sum_{i=1}^{m+1} u_i(f_s(\lambda_i) - p(\lambda_i))^2
= \min_{p\in\Pi_{m-1}} \sum_{i=1}^{m+1} u_i(e(\lambda_i) - p(\lambda_i))^2.
\]
By \cref{lem:divdiff},
\[
\min_{\vec{x}\in\Kry_m}\|f_s(\vec{A})\vec{b}-\vec{x}\|_{\vec{A}^\alpha}^2
= \frac{e[\lambda_1, \ldots,\lambda_{m+1}]^2}{\sum_i (u_i\omega_i^2)^{-1}}
= \frac{c^2}{\left(\sum_i h_i\nu_i\right)^2} \frac{\big(\sum_{i} h_i\, \varphi_{i,s} \big)^2}{\sum_i h_i},
\]
which yields an exact expression for the optimal error. 

\emph{Step 3: Computing the ratio.} Taking the ratio, all constants cancel and we obtain the exact identity
\begin{equation}\label{eq:kantorovich}
\frac{\|f_s(\vec{A})\vec{b}-\vec{f}_m\|_{\vec{A}^\alpha}^2}{\min_{\vec{x}\in\Kry_m}\|f_s(\vec{A})\vec{b}-\vec{x}\|_{\vec{A}^\alpha}^2}
= \frac{\big(\sum_i h_i\,\varphi_{i,s}^2\big)\big(\sum_i h_i\big)}{\big(\sum_i h_i\,\varphi_{i,s}\big)^2} ,
\end{equation}
in which the weights $h_i > 0$ may be prescribed freely.
Consider weights are supported on the two extreme nodes with $h_1 \propto \varphi_{1,s}^{-1}$ and $h_{m+1} \propto \varphi_{m+1,s}^{-1}$, in the right hand side above equals
\begin{equation}\label{eq:cassels_max}
\frac{(\varphi_{1,s} + \varphi_{m+1,s})(\varphi_{1,s}^{-1} + \varphi_{m+1,s}^{-1})}{4} =  \psi(R_s)^2
,\qquad
\psi(R) := \frac{1}{2}\big(R^{1/2} + R^{-1/2}\big),
\quad
R_s := \frac{\varphi_{1,s}}{\varphi_{m+1,s}}.
\end{equation}
Since \cref{lem:prescribe} requires $h_i > 0$ at every node, we take $h_1 = \varphi_{1,s}^{-1}$, $h_{m+1} = \varphi_{m+1,s}^{-1}$, and $h_i = \delta > 0$ at the interior nodes; by continuity of \cref{eq:kantorovich} in the weights, the supremum $\psi(R_s)^2$ is approached as $\delta \to 0$.
Since $R_s = \varphi_{1,s}/\varphi_{m+1,s} = \kappa^{\alpha} \frac{\lambda_{m+1}+s}{\lambda_{1} + s}$, the map $s \mapsto \psi(R_s)$ satisfies $\psi(R_s) = \psi(R_s^{-1})$. Furthermore, $\psi$ is decreasing on $(0,1]$ and increasing on $[1,\infty)$. Hence,  $s \mapsto \psi(R_s)$ reaches its supremum at an endpoint of the range:
\[
\sup_{s \geq 0}\, \psi(R_s) = \max\big\{\psi(\kappa^{1-\alpha}),\, \psi(\kappa^{\alpha})\big\} = \psi(\kappa^E), \quad E = \max\{\alpha,1-\alpha\}.
\]
It is attained at $s = 0$ if $\alpha \leq 1/2$, and is approached as $s \to \infty$ otherwise.
Finally, choose $s \geq 0$ such that  $\psi(R_s) \geq (1-\eta/2)\psi(\kappa^E)$ and then choose $\delta$ sufficiently small that \cref{eq:kantorovich} exceeds $(1-\eta)^2\psi(\kappa^E)^2$. Taking square roots completes the proof.
\end{proof}

\section*{AI statement}

The authors used assistance from AI tools to develop ideas presented in this paper. 
The authors assume responsibility for all content.

\clearpage

\printbibliography

\end{document}